\documentclass[11pt,epsf]{article}
\usepackage{graphicx}
\usepackage{amsthm}
\usepackage{amsfonts}
\usepackage{amssymb}
\title{Some facts about discriminants}
\author{Vladimir Petrov Kostov\\ Universit\'e C\^ote d'Azur, CNRS, LJAD, France
\\  
e-mail: kostov@math.unice.fr} 
\date{}
\newtheorem{tm}{Theorem}

\newtheorem{rem}[tm]{Remark}
\newtheorem{rems}[tm]{Remarks}
\newtheorem{lm}[tm]{Lemma}

\newtheorem{prop}[tm]{Proposition}
\newtheorem{nota}[tm]{Notation}
\newtheorem{st}[tm]{Statement}
\begin{document} 
\maketitle 
\begin{abstract}
For the family of polynomials in one variable 
$P:=x^n+a_1x^{n-1}+\cdots +a_n$ we ask the questions 
at which points its discriminant set can be 
considered as the graph of a function of all coefficients $a_j$ 
but one and how its subset of 
points, where the discriminant set is not smooth, 
projects on the different coordinate 
hyperplanes in the space of the coefficients~$a$.\\ 

{\bf AMS classification:} 12E05; 12D05\\

{\bf Key words:} polynomial in one variable; discriminant set; 
resultant; multiple root
\end{abstract}

\section{Introduction}

Consider the family of polynomials $P:=x^n+a_1x^{n-1}+\cdots +a_n$ 
in the variable $x\in \mathbb{C}$. Set $a:=(a_1,\ldots ,a_n)$ 
and $a^k:=(a_1,\ldots ,a_{k-1},a_{k+1},\ldots ,a_n)$. Denote by 
$D:=\{ a\in \mathbb{C}^n|$Res$(P,P',x)=0\}$ 
its {\em discriminant set}, i.~e. the set of values of $a$ for which $P$ has a 
multiple root. In the present text we treat the question at which 
points the set $D$ can be 
considered as the graph of a function in the variables $a^j$ 
(and for which $j$) and how the set of its 
points, where it is not smooth, projects on the different coordinate 
hyperplanes in the space of the variables~$a$. The results of this paper 
are applicable to the case $x$, $a_j\in \mathbb{R}$, when it is 
important to know the number of positive and negative real roots and the 
signs of the coefficients of $P$, see \cite{AF}, \cite{FKS} and the references 
therein. Some recent results about real 
discriminant sets can be found in \cite{Ko}.
 
The following result is known; we include its proof for the 
sake of completeness:

\begin{lm}\label{lm1}
The set $D$ is the zero set of an irreducible polynomial 
$R$ in $a\in \mathbb{C}^n$. When considered as a polynomial in 
$a_k$, $k\leq n-1$, it is 
of degree $n$, with leading coefficient $\pm k^k(n-k)^{n-k}a_n^{n-k-1}$. The 
polynomial $R$ is degree $n-1$ in $a_n$. One has $R|_{a_{n-1}=a_n=0}\equiv 0$.
\end{lm}

\begin{rem}\label{rem1}
{\rm One can assign $j$ as quasi-homogeneous weight to the coefficient $a_j$, 
$1\leq j\leq n$ (because $a_j$ is a symmetric degree $j$ polynomial 
in the roots of $P$). It is well-known that $R$ is a quasi-homogeneous 
polynomial in $a$ of quasi-homogeneous degree $n(n-1)$. It equals 
$\pm \prod _{1\leq i<j\leq n}(x_i-x_j)^2$, 
where $x_i$, $x_j$ are the roots of $P$. In this 
product each difference is squared because when the coefficients 
$a_j$ take values corresponding to a loop in the space $\mathbb{C}^n$ 
circumventing the zero set of $R$, then generically two of the roots are 
interchanged.}
\end{rem}

\begin{nota}\label{notanota}
{\rm (1) We denote by 
$\Sigma$ 
the subset of $D$ on which 
the corresponding polynomial $P$ has a root of 
multiplicity at least $3$ and by 
$\Sigma _k$ its projection  
in the space $\mathbb{C}_k^{n-1}$ of the variables $a^k$. 
Derivations w. r. t. $x$ and $a_k$ 
are denoted respectively by $'$ and $\partial /\partial a_k$.

(2) We denote by $S(F_1,F_2)$ 
the {\em Sylvester matrix} of the polynomials 
$F_1=d_0x^{n_1}+d_1x^{n_1-1}+\cdots +d_{n_1}$ and 
$F_2=g_0x^{n_2}+g_1x^{n_2-1}+\cdots +g_{n_2}$ (considered as polynomials 
in the variable $x$). The matrix  
$S(F_1,F_2)$ is $(n_1+n_2)\times (n_1+n_2)$ and its first (resp.  
its $(n_2+1)$st) row equals

$$(d_0,d_1,\ldots ,d_{n_1},0,\ldots ,0)~~{\rm ,~resp.}~~
(g_0,g_1,\ldots ,g_{n_2},0,\ldots ,0)~,$$
its second and $(n_2+2)$nd rows are obtained by shifting these ones 
by one position to the right while adding a zero in the first position etc. 
For polynomials $G_1$, $G_2$ in $a_k$, with coefficients in $\mathbb{C}[a^k]$, 
we write $S(G_1,G_2,a_k)$ for their Sylvester matrix.  

(3) We denote by $D_k$ the subset in the space $\mathbb{C}_k^{n-1}$ 
of the variables $a^k$ defined by the condition 
$\tilde{D}_k:=$Res$(R,\partial R/\partial a_k,a_k)=0$. When $R$ is considered 
as a polynomial in the variable $a_k$ with 
coefficients in $\mathbb{C}[a^k]$, the set $D_k$ is the subset of 
$\mathbb{C}_k^{n-1}$ on which $R$ has a multiple root.}
\end{nota}

\begin{proof}[Proof of Lemma~\ref{lm1}]
Consider the matrix $S(P,P')$, see Notation~\ref{notanota}. To simplify the 
computation of its determinant we subtract, for $j=1,\ldots ,n-1$, its 
$(n-1+j)$th row multiplied by $1/(n-k)$ from the $j$th one. We denote 
by $S^*$ the newly obtained matrix. The only product 
of $2n-1$ entries in the determinant of $S^*$ which contains 
$n$ factors $a_k$ is the following one 
(we list to the left the entries and to the right, for each entry, 
the positions in which it is encountered):

$$\begin{array}{rcllcl}
-k/(n-k)&~~~~~~~~~~&(1,1)&(2,2)&\ldots&(k,k)\\ \\ 

a_n&&(k+1,k+n+1)&(k+2,k+n+1)&\ldots&(n-1,2n-1)\\ \\ 

(n-k)a_k&&(n,k+1)&(n+1,k+2)&\ldots&(2n-1,k+n) 
\end{array}$$
Up to a sign the product equals $k^k(n-k)^{n-k}a_n^{n-k-1}a_k^n$. 

The matrix $S(P,P')$ contains $n-1$ terms $a_n$, in positions $(j,j+n)$, 
$j=1$, $\ldots$, $n-1$. The product of these terms and of the constant 
terms in positions $(j+n-1,j)$, $j=1$, $\ldots$, $n$ give the only 
monomial $Aa_n^{n-1}$, $A\neq 0$, in $\det S(P,P')$. Irreducibility of $R$ 
follows from the fact that a quasi-homogeneous polynomial with 
quasi-homogeneous weight $j$ of 
the variable $a_j$ (see Remark~\ref{rem1}) 
and containing monomials $Aa_n^{n-1}$ and $Ba_{n-1}^n$, 
$B\neq 0$, cannot be represented as a product of two nonconstant 
quasi-homogeneous polynomials. The equality $R|_{a_{n-1}=a_n=0}\equiv 0$ follows 
from $S(P,P')$ having as nonzero entries in its last column only $a_n$ 
and $a_{n-1}$.
\end{proof}

\section{Where is the discriminant locally the graph 
of a function?}

\begin{tm}\label{prop1}
(1) Suppose that at a point $A\in D$ the corresponding polynomial $P$ has a 
double root $\lambda$ and $n-2$ simple roots. Then 

(i) if in addition $\lambda \neq 0$ (hence one does not have $a_{n-1}=a_n=0$), 
then for $k=1,\ldots ,n$, 
at this point the set $D$ is locally the graph 
of an analytic function in $a^k$;

(ii) if $\lambda =0$, then this property holds true only for $k=n$ and fails 
for $k=1,\ldots ,n-1$.

(2) Suppose that at a point $A\in D$ the corresponding polynomial $P$ has a 
root of multiplicity at least $3$. Then at this point the hypersurface $D$ 
is not smooth. 
\end{tm}

\begin{proof}[Proof of Theorem~\ref{prop1}]
Set $P=(x+\lambda )^2Q$, 
$Q:=(x^{n-2}+b_1x^{n-3}+\cdots +b_{n-2})$. Hence 

$$\begin{array}{lll}
a_1=2\lambda +b_1~~,~~&a_2=\lambda ^2+2\lambda b_1+b_2~,&
a_3=\lambda ^2b_1+2\lambda b_2+b_3~,~\ldots ~,\\ \\ 
a_{n-2}=\lambda ^2b_{n-4}+2\lambda b_{n-3}+b_{n-2}~,&
a_{n-1}=\lambda ^2b_{n-3}+2\lambda b_{n-2}~,&
a_n=\lambda ^2b_{n-2}~.\end{array}$$
Consider the $(n-1)\times n$ matrix 
$\tilde{J}:=(\partial (a_1,a_2,\ldots ,a_n)/\partial 
(\lambda ,b_1,\ldots ,b_{n-2}))^T$. It equals 

$$\left( \begin{array}{cccccccc}
2&2\lambda +2b_1&2\lambda b_1+2b_2&2\lambda b_2+2b_3&
\cdots &2\lambda b_{n-4}+2b_{n-3}&
2\lambda b_{n-3}+2b_{n-2}&2\lambda b_{n-2}\\ \\ 
1&2\lambda&\lambda ^2&0&\cdots&0&0&0\\ \\ 
0&1&2\lambda&\lambda ^2&\cdots&0&0&0\\ \\ 
0&0&1&2\lambda&\cdots&0&0&0\\ \\
\vdots&\vdots&\vdots&\vdots&\ddots&\vdots&\vdots&\vdots \\ \\ 
0&0&0&0&\cdots&1&2\lambda&\lambda ^2\end{array}\right)$$
We denote by $C_k$ the $k$th column of the matrix $\tilde{J}$ and by 
$J_k$ its submatrix obtained by deleting $C_k$.  
As $Q(-\lambda )\neq 0$, all claims of the 
theorem follow from the following lemma:

\begin{lm}\label{lm2}
One has $\det J_k=(-1)^n2\lambda ^{n-k}Q(-\lambda )$.
\end{lm}
\end{proof}

\begin{proof}[Proof of Lemma~\ref{lm2}]
Denote by $S_k$ the $(n-1)$-vector-column 
$(2b_{k-1},0,\ldots ,0,\lambda ,1,0,\ldots ,0)^T$,  
where $2\leq k\leq n-1$ and $\lambda$ is preceded by $k-2$ zeros. 
We set $S_1:=(2,1,0,\ldots ,0)^T$, $S_0:=(0,\ldots ,0)^T$. 
Hence $C_k=\lambda S_{k-1}+S_k$. Set $J_k:=J_k'+J_k''$, where 

$$J_k'=\lambda (C_1,\ldots ,C_{k-1},S_k,C_{k+2},\ldots ,C_n)~~,~~
J_k''=(C_1,\ldots ,C_{k-1},S_{k+1},C_{k+2},\ldots ,C_n)~.$$
Thus $\det J_k=\lambda \det J_k'+\det J_k''$. One has $\det J_k''=0$. Indeed, 

$$\det J_k''=\det (C_1,\ldots ,C_{k-1},S_{k+1},\lambda S_{k+1}+S_{k+2},
\lambda S_{k+2}+S_{k+3},\ldots ,\lambda S_{n-2}+S_{n-1},\lambda S_{n-1})~.$$
Subtract consecutively for $j=k$, $\ldots$, $n-2$, 
the $j$th column multiplied by $\lambda$ from the $(j+1)$st one. This 
does not change the determinant. After these subtractions 
all entries of the last column 
are zeros, so $\det J_k''=0$ and $\det J_k=\lambda \det J_k'$. 
After this set $J_k':=J_k^*+J_k^{**}$, where 

$$J_k^*=\lambda ^2(C_1,\ldots ,C_{k-1},S_k,S_{k+1},C_{k+3},\ldots ,C_n)~~,~~
J_k^{**}=\lambda (C_1,\ldots ,C_{k-1},S_k,S_{k+2},C_{k+3},\ldots ,C_n)~.$$
By analogy with $\det J_k''=0$ we show that $\det J_k^{**}=0$. 
Hence $\det J_k=\lambda ^2\det J_k^*$. 
Continuing like this we conclude that 

$$\begin{array}{ccl}\det J_k&=&
\lambda ^{n-k}\det (C_1,\ldots ,C_{k-1},S_k,S_{k+1},\ldots ,S_{n-1})\\ \\ 
&=&\lambda ^{n-k}\det (S_1,\lambda S_1+S_2,\ldots ,\lambda S_{k-2}+S_{k-1},
S_k,S_{k+1},\ldots ,S_{n-1})~.\end{array}$$
We subtract then consecutively 
(for $j=1,\ldots ,k-2$) the $j$th column multiplied by 
$\lambda$ from the $(j+1)$st one. This does not change $\det J_k$,~so

$$\det J_k=\lambda ^{n-k}\Delta ~~,~~\Delta :=\det (S_1, S_2,\ldots ,S_{k-1},
S_k,S_{k+1},\ldots ,S_{n-1})~~.~~{\rm i.~e.}$$

$$\Delta =\left| \begin{array}{ccccccc}
2&2b_1&2b_2&2b_3&\cdots &2b_{n-3}&
2b_{n-2}\\ \\ 
1&\lambda&0&0&\cdots&0&0\\ \\ 
0&1&\lambda&0&\cdots&0&0\\ \\ 
0&0&1&\lambda&\cdots&0&0\\ \\
\vdots&\vdots&\vdots&\vdots&\ddots&\vdots&\vdots \\ \\ 
0&0&0&0&\cdots&1&\lambda\end{array}\right|$$ 
To compute $\Delta$ we subtract consecutively (for $j=1,\ldots ,n-2$) 
the $j$th column multiplied by $\lambda$ from the $(j+1)$st one. We get 

$$\Delta =2\left| \begin{array}{cccccc}
1&b_1-\lambda&b_2-b_1\lambda +\lambda ^2&\cdots &
b_{n-3}-\lambda b_{n-4}+\cdots +(-1)^{n-3}\lambda ^{n-3}&
Q(-\lambda )\\ \\ 
1&0&0&\cdots&0&0\\ \\ 
0&1&0&\cdots&0&0\\ \\ 
\vdots&\vdots&\vdots&\ddots&\vdots&\vdots \\ \\ 
0&0&0&\cdots&1&0\end{array}\right|$$ 
Hence $\Delta =(-1)^n2Q(-\lambda )$. 
\end{proof}

\begin{rems}
{\rm (1) Consider the product $P_1P_2$, where $P_1:=x^m+b_1x^{m-1}+\cdots +b_m$, 
$P_2:=x^l+c_1x^{l-1}+\cdots +c_l$ and $b_i$, $c_j$ are complex parameters. 
Suppose that for some value of these parameters the polynomials $P_1$ and $P_2$ 
have no root in common. Then locally (close to this value) the discriminant set 
of $P_1P_2$ (defined by analogy with $D$) is diffeomorphic to the direct 
product of the discriminant sets of $P_1$ and $P_2$ at the respective values of 
$b_i$ and $c_j$. This follows from the Lemma about the product 
on p. 12 of~\cite{Me}. (In \cite{Me} the author considers the case of 
real polynomials, but the proof of the Lemma about the product is carried out 
in the complex case in exactly the same way as in the real one.) 

(2) Suppose that at a point $A\in D$ the polynomial $P$ has $k$ double roots 
and $n-2k$ simple ones. Then locally, at $A$, the set $D$ is the 
transversal intersection of $k$ analytic hypersurfaces each of which 
satisfies statements (i) and (ii) of Proposition~\ref{prop1}. Transversality 
follows from part (1) of the present remarks, the analogs of 
statements (i) and (ii) of 
the proposition are proved in exactly the same way as the proposition itself 
(in the proof of the proposition we do not use the fact 
that the rest of the roots of $P$ are simple).}
\end{rems}

\begin{prop}\label{lmdivisible}
(1) For $k\leq n-1$ the polynomial $\tilde{D}_k$ is not divisible by 
$a_i$ for $k\neq i\neq n$. 

(2) The polynomial $\tilde{D}_n$ is not divisible by 
$a_i$ for $i\leq n-1$. 

(3) If $k\leq n-2$, then 
$\tilde{D}_k$ is divisible by $a_n^{n-k-1}$ and not divisible by 
$a_n^{n-k}$. 

(4) The polynomial $\tilde{D}_{n-1}$ is divisible by $a_n$ and not divisible 
by $a_n^2$.
\end{prop}

\begin{proof}[Proof of Proposition~\ref{lmdivisible}]
Throughout the proof the letter $\Omega$ (indexed or not) 
stands for nonspecified 
nonzero constants. We set $T_k:=S(R,\partial R/\partial a_k,a_k)$. Recall that 
$T_k$ is $(2n-1)\times (2n-1)$ for $k\leq n-1$ and $(2n-3)\times (2n-3)$ 
for $k=n$, see Lemma~\ref{lm1}.   

\begin{st}\label{A}
For $a_i=0$, $k\neq i\neq n$, $k<n$, one has 
$R=\Omega _1a_k^na_n^{n-k-1}+\Omega _2a_n^{n-1}$.
\end{st} 

\begin{proof}
Indeed, in this case one computes $R$ easily if one subtracts for 
$j=1$, $\ldots$, $n-1$ the $(n-1+j)$th row of $S(P,P')$ 
multiplied by $1/(n-k)$ from 
its $j$th one. This doesn't change $\det S(P,P')=R$ 
and the matrix obtained from $S(P,P')$ has only the following nonzero entries: 

$$\begin{array}{ccccccc}
-k/(n-k)&&(j,j)&,&a_n&&(j,j+n)\\
&{\rm in~positions}&&&&{\rm in~positions}&\\ 
n&&(n-1+\nu ,\nu )&,&(n-k)a_k&&(n-1+\nu ,\nu +k)\end{array}$$
($1\leq j\leq n-1$, $1\leq \nu \leq n$). To prove Statement~\ref{A} with 
this form of $\det S(P,P')$ 
is easy. 
\end{proof}

Statement~\ref{A} implies that for $a_i=0$, $k\neq i\neq n$, $k<n$, the matrix 
$T_k$ has nonzero entries only 
$\Omega _1a_n^{n-k-1}$, $\Omega _2a_n^{n-1}$ and $n\Omega _1a_n^{n-k-1}$, 
respectively in positions $(j,j)$, $(j,j+n)$, $j=1,\ldots ,n-1$, and 
$(n-1+\nu ,\nu )$, 
$\nu =1,\ldots ,n$. Clearly 
$\det T_k=\Omega _3a_n^{(n-1)^2+n(n-k-1)}\not\equiv 0$. Hence 
$\tilde{D}_k$ is not divisible by $a_i$ for $i\neq n$. Part (1) is proved.

To prove part (2) for $i<n-1$ we use Statement~\ref{A} with $k=n-1$. Hence 
the $(2n-3)\times (2n-3)$-matrix 
$T_{n-1}|_{a_i=0,n-1\neq i\neq n}$ has nonzero entries only 

$$\begin{array}{ccccccc}
\Omega _2&&(j,j)&,&\Omega _1a_{n-1}^n&{\rm in~positions}&(j,n-1+j)\\
&{\rm in~positions}&&&&&\\ 
(n-1)\Omega _2&&(n-2+\nu ,\nu )&&&&
\end{array}$$
($1\leq j\leq n-2$, $1\leq \nu \leq n-1$). It is easy to show 
that its determinant equals $\Omega _4a_{n-1}^{n(n-2)}\not\equiv 0$. 

To prove part (2) for $i=n-1$ we apply Statement~\ref{A} with $k=n-2$. 
Hence $T_n$ has nonzero entries only 
(with $j$ and $\nu$ as above) 

$$\begin{array}{ccccccc}
\Omega _2&&(j,j)&,&\Omega _1a_{n-2}^n&&
(j,n-2+j)\\
&{\rm in~positions}&&&&{\rm in~positions}&\\ 
(n-1)\Omega _2&&(n-2+\nu ,\nu )&,&\Omega _1a_{n-2}^n&&(n-2+\nu ,n-2+\nu )
\end{array}$$
Its determinant equals $\Omega _5a_{n-2}^{n(n-1)}\not\equiv 0$. 
Part (2) is proved. 

To prove part (3) consider the $(2n-1)\times (2n-1)$-matrix $T_k$. 
It has two entries in its first column, in positions $(1,1)$ and $(n,1)$. 
By Lemma~\ref{lm1} they are of the form $\Omega a_n^{n-k-1}$ and 
$n\Omega a_n^{n-k-1}$. This 
proves the first statement of part (3).
To prove its second statement consider for $k\leq n-2$ 
the polynomial $P^0:=P|_{a_i=0,k\neq i\neq n-1}=x^n+a_kx^{n-k}+a_{n-1}x$. 
Set $\delta :=\det S(P^0,{P^0}')$, $\delta _1:=\det S(P^0/x,{P^0}')$.  

\begin{st}\label{B}
One has 
$\delta =a_{n-1}\delta _1=\Omega _5a_{n-1}^{n}+\Omega _6a_k^{n-1}a_{n-1}^{n-k}$.
\end{st}

\begin{proof}
In its last column 
the matrix $S(P^0,{P^0}')$ has 
a single nonzero entry (namely $a_{n-1}$, in position $(2n-1,2n-1)$), 
so $\delta =a_{n-1}\delta _1$. 
To compute easily 
$\delta _1$ 
we subtract for $j=n$, $n+1$, $\ldots$, $2n-2$ the $j$th row of 
$S(P^0/x,{P^0}')$ 
from its $(j-n+1)$st row. This makes disappear the terms $a_{n-1}$ in the first 
$n-1$ rows. After the subtractions the first $n-1$ rows have entries 
$\Omega _*$ in positions $(j,j)$, $\Omega 'a_k$ in positions $(j,j+k)$ 
and zeros elsewhere. 

Then we subtract for $i=1,\ldots ,n-1$ the $i$th row multiplied by a 
constant from the $(i+n-1)$st one. The constant is chosen such that after the 
subtraction the $(i+n-1)$st row contains no term $\Omega ^*a_k$. 
After all these subtractions the matrix obtained from $S(P^0/x,{P^0}')$ has 
nonzero terms in the following positions (and zeros elsewhere):

$$\begin{array}{ccccccc}
\Omega _*&&(j,j)&,&\Omega 'a_k&&
(j,j+k)\\
&{\rm in~positions}&&&&{\rm in~positions}&\\ 
\Omega _{**}&&(j+n-1,j)&,&\Omega _{***}a_{n-1}&&
(j+n-1,j+n-1) 
\end{array}$$
($1\leq j\leq n-1$). 
Hence $\delta =a_{n-1}\delta _1=\Omega _5a_{n-1}^{n}+\Omega _6a_k^{n-1}a_{n-1}^{n-k}$. 
\end{proof}

Thus the matrix $T_k|_{a_i=0,k\neq i\neq n-1}$ 
has only the following nonzero entries:

$$\begin{array}{ccccccc}
\Omega _6a_{n-1}^{n-k}&&(j,j+1)&,&\Omega _5a_{n-1}^n&{\rm in~positions}&
(j,n+j)\\
&{\rm in~positions}&&&&&\\ 
(n-1)\Omega _6a_{n-1}^{n-k}&&(n-1+\nu ,\nu +1)&,&1\leq j\leq n-1&~~~,~~~&
1\leq \nu \leq n~. 
\end{array}$$

We consider the expansion of $\det T_k$ in a series in $a_n$. Our aim is 
to show that the initial term equals $Wa_n^{n-k-1}$ with $W\not\equiv 0$ with 
which part (3) will be proved. When expanding $\det T_k$ w.~r.~t. its 
first column (where it has just two nonzero entries) one gets 

$$\det T_k=(-1)^{1+1}\Omega a_n^{n-k-1}\det (T_k)_{1,1}+
(-1)^{n+1}n\Omega a_n^{n-k-1}\det (T_k)_{n,1}~,$$
where $(T_k)_{p,q}$ means the matrix obtained from $T_k$ by deleting its $p$th 
row and its $q$th column. In the first column of the matrix $(T_k)_{1,1}$ 
(resp. $(T_k)_{n,1}$) only 
the entry in position $(n-1,1)$ (resp. $(1,1)$) does not vanish for $a_n=0$, 
see Lemma~\ref{lm1} and Statement~\ref{B}. 
For $a_i=0$, $k\neq i\neq n-1$, this entry equals 
$(n-1)\Omega _6a_{n-1}^{n-k}$ (resp. 
$\Omega _6a_{n-1}^{n-k}$, see Statement~\ref{B}). Thus 

$$\begin{array}{cccc}
\det (T_k)_{1,1}&=&
(-1)^{n-1+1}(n-1)\Omega _6a_{n-1}^{n-k}(\det ((T_k)_{1,1})_{n-1,1}|_{a_n=0})
+O(\sum _{k\neq i\neq n-1}|a_i|)&,\\ \\ 
\det (T_k)_{n,1}&=&
(-1)^{1+1}\Omega _6a_{n-1}^{n-k}(\det ((T_k)_{n,1})_{1,1}|_{a_n=0})+
O(\sum _{k\neq i\neq n-1}|a_i|)&.
\end{array}$$ 
Now observe that each of the matrices 

$$((T_k)_{1,1})_{n-1,1}|_{a_i=0, k\neq i\neq n-1}~~~{\rm and}~~~ 
((T_k)_{n,1})_{1,1}|_{a_i=0, k\neq i\neq n-1}$$ 
is equal to the Sylvester matrix 
$S(~\Omega _5a_{n-1}^{n}+\Omega _6a_k^{n-1}a_{n-1}^{n-k}~,~
\Omega _6(n-1)a_k^{n-2}a_{n-1}^{n-k}~,~a_k~)=:S'$ whose determinant is 
of the form $\Omega ^{***}a_{n-1}^{n(n-2)+(n-k)(n-1)}\not\equiv 0$, 
therefore 
$$\begin{array}{ccccc}
W_{a_i=0, k\neq i\neq n-1}&=&\Omega \Omega _6a_{n-1}^{n-k}
\left( (-1)^{1+1}(-1)^{n-1+1}(n-1)+
(-1)^{n+1}(-1)^{1+1}n\right) \det S'&\\ \\ 
&=&\Omega \Omega _6a_{n-1}^{n-k}(-1)^{n-1}\det S'&\not\equiv &0~.\end{array}$$  
This proves part (3). To prove part (4) we need  

\begin{st}\label{C} The polynomial $R$ is of the form 
$a_{n-1}^2U(a)+a_nV(a)$, where $U$ and $V|_{a_n=0}\not\equiv 0$ 
are polynomials.
\end{st}

\begin{proof}
Set $S^0:=S(P,P')|_{a_n=0}$. The only nonzero entry in the last column of 
$S^0$ is 
$a_{n-1}$ in position $(2n-1,2n-1)$, so 
$R|_{a_n=0}=a_{n-1}\det ((S^0)_{2n-1,2n-1})$. 
Both nonzero entries in the last column of $(S^0)_{2n-1,2n-1}$ equal 
$a_{n-1}$, so $R|_{a_n=0}$ is divisible by $a_{n-1}^2$ hence 
$R=a_{n-1}^2U(a)+a_nV(a)$. There remains to show that 
$V|_{a_n=0}\not\equiv 0$. To this end consider the matrix 
$M:=S(P,P')|_{a_{n-1}=0}$. One has $\det M=(-1)^{n-1+2n-1}a_n\det ((M)_{n-1,2n-1})$. 
For $a_i=0$, $i\neq n-2$, the only nonzero entries of $(M)_{n-1,2n-1}$ are 

$$\begin{array}{ccccccc}
1&&(j,j)&,&a_{n-2}&&
(j,n-2+j)\\
&{\rm in~positions}&&&&{\rm in~positions}&\\ 
n&&(n-2+\nu ,\nu )&,&2a_{n-2}&&(n-2+\nu ,n-2+\nu )
\end{array}$$
($1\leq j\leq n-2$, $1\leq \nu \leq n-1$). Hence 
$\det ((M)_{n-1,2n-1})|_{a_i=0,i\neq n-2})=\Omega ^{\omega}a_{n-2}^{n-1}\not\equiv 0$.
\end{proof}

Statement \ref{C} implies that the last column of 
$T_{n-1}|_{a_n=0}$ contains only zeros, so  
$\tilde{D}_{n-1}$ is divisible by $a_n$. Denote by $t_{i,j}$ the entry of 
$T_{n-1}$ in position $(i,j)$. Hence 

$$\det T_{n-1}=
(-1)^{3n-5}t_{n-2,2n-3}\det ((T_{n-1})_{n-2,2n-3})+
(-1)^{4n-6}t_{2n-3,2n-3}\det ((T_{n-1})_{2n-3,2n-3})$$
The terms $t_{n-2,2n-3}$ and $t_{2n-3,2n-3}$ are divisible by $a_n$ and 
$t_{n-2,2n-3}=a_nV(a)$ is not divisible by $a_n^2$. 
All entries in the last column of 
$(T_{n-1})_{2n-3,2n-3}$ are divisible by $a_n$ (see Statement~\ref{C}), so 
$t_{2n-3,2n-3}\det ((T_{n-1})_{2n-3,2n-3})=O(a_n^2)$. 
On the other hand the nonzero entries of the matrix 
$\tilde{T}:=(T_{n-1})_{n-2,2n-3}|_{a_i=0,i\neq n-2, n-1}$ are 
(see Statement~\ref{B} with $k=n-2$)

$$\begin{array}{ccccccc}
\Omega _5&&(j,j)&,&\Omega _6a_{n-2}^{n-1}&&
(j,n-2+j)\\
&{\rm in~positions}&&&&{\rm in~positions}&\\ 
n\Omega _5&&(n-2+\nu ,\nu )&,&2\Omega _6a_{n-2}^{n-1}&&(n-2+\nu ,n-2+\nu )
\end{array}$$
($1\leq j\leq n-2$, $1\leq \nu \leq n-1$). Hence 
$\det \tilde{T}=\Omega _{\lambda}a_{n-2}^{(n-1)^2}\not\equiv 0$. 
Thus $t_{n-2,2n-3}\det ((T_{n-1})_{n-2,2n-3})$ 
(and hence $\det T_{n-1}$ as well) is divisible by $a_n$, 
but not by $a_n^2$. This proves 
part~(4).
\end{proof}

\section{The projections of the set $\Sigma$}

The set $\Sigma$ (see Notation~\ref{notanota}) is a codimension $2$ algebraic 
subset of $\mathbb{C}^n$. It is locally defined 
by the equations of any two of its projections $\Sigma _k$, 
or by Res$(P,P',x)=$Res$(P',P'',x)=0$. 
Consider for $1\leq k\leq n-1$ the polynomial $P_k:=P-xP'/(n-k)$; 
its coefficient of $x^{n-k}$ equals $0$. 
When $P$ has a root $\alpha \neq 0$ of 
multiplicity $m\geq 1$, then $\alpha$ is a root of $P_k$ of multiplicity 
$m-1$. We set $P_n:=P'$.  

\begin{tm}\label{prop2}
For $k\neq n-1$ the polynomial $V_k:=${\rm Res}$(P_k,P_k',x)$ is irreducible. 
The polynomial {\rm Res}$(P_{n-1},P_{n-1}',x)$ is the product of $a_n$ 
and of an irreducible polynomial in~$a^{n-1}$; we set 
$V_{n-1}:=${\rm Res}$(P_{n-1},P_{n-1}',x)/a_n$.
\end{tm}

\begin{tm}\label{propsets}
(1) For $k=1,\ldots ,n$  
the set $\Sigma _k$ is defined by the condition 
$V_k=0$. 

(2) The sets $\Sigma _k$ are irreducible.
\end{tm}

\begin{proof}[Proof of Theorem~\ref{prop2}]
Irreducibility of $V_n$ follows from an analog of Lemma~\ref{lm1} formulated 
for $P'$ instead of $P$. 
For $1\leq k\leq n-2$ 
the polynomial $V_k$  
is a quasi-homogeneous polynomial in $a^k$ containing monomials 
$A_ka_n^{n-1}$ and $B_ka_{n-1}^n$ (resp. $A_na_{n-1}^{n-2}$ and $B_na_{n-2}^{n-1}$), 
where $A_k\neq 0$ and $B_k\neq 0$ are constants, see Lemma~\ref{lm1} 
and its proof. 
Such a polynomial cannot be represented as a product of two quasi-homogeneous 
polynomials of smaller quasi-homogeneous degrees because the quasi-homogeneous 
weights of $a_n$, $a_{n-1}$ (resp. of $a_{n-1}$ and $a_{n-2}$) equal $n$ and 
$n-1$ (resp. $n-1$ and $n-2$). (If such a representation exists, and if 
$h_1$ and $h_2$ are the quasi-homogeneous degrees of the two factors, then 
the monomial $A_ka_n^{n-1}$ is a product of monomials 
$H^1a_n^{k_1}$ and $H^2a_n^{k_2}$ of each factor and $k_1n=h_1$, $k_2n=h_2$, 
$k_1<n-1$, $k_2<n-1$. 
In the same way the monomial $B_ka_{n-1}^n$ 
implies the existence of $l_1, l_2\in \mathbb{N}$ such that $l_1(n-1)=h_1=k_1n$, 
$l_2(n-1)=h_2=k_2n$, $l_1<n$, $l_2<n$ which is impossible.)    

The polynomial $R_{n-1}:=$Res$(P_{n-1},P_{n-1}',x)$ 
is reducible. This follows from 
\vspace{1mm}

{\bf Property~A}. {\em The nonzero entries in the last two columns of 
the matrix $S:=S(P_{n-1},P_{n-1}')$ are $s_{n-2,2n-2}=s_{n-1,2n-1}=a_n$ and 
$s_{2n-1,2n-2}=-2a_{n-2}$.}
\vspace{1mm}

Hence every monomial of $R_{n-1}$ is divisible by $a_n$, and 
$R_{n-1}/a_n$ (of quasi-homogeneous degree $n(n-2)$) 
contains monomials of the form $C_ja_j^na_n^{n-j-2}$, $C_j\neq 0$, 
$j=1,2,\ldots$, $n-2$, $n$  
(this is proved by complete analogy with Lemma~\ref{lm1}). If $n$ is odd, then 
$(n,n-2)=1$ and applying to the monomials $C_{n-2}a_{n-2}^n$ and 
$C_na_n^{n-2}$ a reasoning similar to the one above (which was 
applied to $A_ka_n^{n-1}$ and 
$B_ka_{n-1}^n$) one concludes that $R_{n-1}/a_n$ is irreducible. 

If $n$ is even, then $R_{n-1}/a_n$ could be reducible only if it is 
the product of two polynomials of quasi-homogeneous degree 
$n(n-2)/2$. To show that this is impossible consider the monomial 
$C_{n-3}a_{n-3}^na_n$ (for which $C_{n-3}\neq 0$, see Lemma~\ref{lm1}). 
It must be the product of monomials $C_{n-3}'a_{n-3}^s$ and $C_{n-3}''a_{n-3}^qa_n$. 
Hence $s(n-3)=n(n-2)/2=q(n-3)+n$, i.e. $s=n(n-2)/2(n-3)$ and 
$q=n(n-4)/2(n-3)$. The numbers $n-2$ and $n-3$ are coprime, and such are 
$n-4$ and $n-3$ 
as well. For $n>6$ the ratio $n/(n-3)$ is not integer, so such a 
product of polynomials (equal to $R_{n-1}/a_n$) does not exist. In the 
particular cases $n=4$ and $n=6$ one can check with the help of a computer 
that the corresponding polynomials $P_3/a_4$ and $P_5/a_6$ are irreducible.
\end{proof}

\begin{proof}[Proof of Theorem~\ref{propsets}]
For $k\neq n-1$ the polynomial $V_k$ defines the subset 
$Z_k\subset \mathbb{C}_k^{n-1}$ on which the polynomial $P_k$ 
has a multiple root. Hence 
$\Sigma _k\subset Z_k$. On the other hand the set $Z_k$ is irreducible 
(see Theorem~\ref{prop2}) and of 
codimension $1$ in the space $\mathbb{C}_k^{n-1}$, therefore $\Sigma _k=Z_k$.  
For $k\neq n-1$ part (2) follows from the irreducibility 
of the polynomials $V_k$, see Theorem~\ref{prop2}. 

Let $k=n-1$. The polynomial $V_{n-1}$ is of the form $a_{n-2}W^*+a_nW^{**}$, 
$W^*,W^{**}\in \mathbb{C}[a^{n-1}]$, see Property~A. If $0$ is a triple root 
of $P$, then $a_{n-2}=a_{n-1}=a_n=0$ and the projection of this point of $\Sigma$ 
in the space $\mathbb{C}_{n-1}^{n-1}$ belongs to the set $\{ V_{n-1}=0\}$. 
If $\beta \neq 0$ is a triple root of $P$, then 
$P=(x-\beta )^3(x^{n-3}+g_1x^{n-4}+\cdots +g_{n-3})$, where the coefficients 
$g_j$ run over a whole neighbourhood in $\mathbb{C}^{n-3}$. Such points of 
$\Sigma$ project in $\mathbb{C}_{n-1}^{n-1}$ in the set $\{ V_{n-1}=0\}$, but not 
in $\{ a_n=0\} \backslash \{ V_{n-1}=0\}$. This proves the theorem for 
$k=n-1$.   
\end{proof}

\end{document}